\documentclass[11pt,twoside]{preprint}

\usepackage{hyperref}
\usepackage{breakurl}

\usepackage{times}
\usepackage{microtype}
\usepackage{amsmath}
\usepackage{amssymb}
\usepackage{tikz}
\usepackage{wasysym}
\usepackage{centernot}
\usepackage{color}

\usepackage{mathrsfs}
\usepackage{mhequ}
\usepackage{mhenvs}
\usepackage{mhsymb}
\usepackage[top=4cm, bottom=3.5cm, left=3.5cm, right=3.2cm]{geometry}

\colorlet{symbols}{blue!90!black}
\def\symbol#1{\textcolor{symbols}{#1}}
\def\1{\mathbf{\symbol{1}}}

\colorlet{testcolor}{green!60!black}

\definecolor{darkred}{rgb}{0.9,0.1,0.1}

\def\comment#1{\ifthenelse{\isodd{\value{page}}}{\marginpar{\raggedright\scriptsize{\textcolor{darkred}{#1}}}}{\marginpar{\raggedleft\scriptsize{\textcolor{darkred}{#1}}}}}  

\newop{diag}

\def\EE{{\mathscr E}}
\def\FF{{\mathscr F}}

\def\s{{\mathrm{s}}}
\def\e{{\mathrm{e}}}
\def\m{{\mathfrak{m}}}

\def\limsup{\mathop{\overline{\mathrm{lim}}}}
\def\liminf{\mathop{\underline{\mathrm{lim}}}}

\def\${|\!|\!|}
\def\l|{\left|\!\left|\!\left|}
\def\r|{\right|\!\right|\!\right|}

\begin{document}

\title{Snapping out Walsh's Brownian motion and related stiff problem}
\author{Liping Li$^1$, Wenjie Sun$^2$}
\institute{RCSDS, HCMS, Academy of Mathematics and Systems Science, Chinese Academy of Sciences, \email{liliping@amss.ac.cn} \and Fudan University, \email{wjsun14@fudan.edu.cn}}

\maketitle

\begin{abstract}
Firstly, we shall introduce the so-called snapping out Walsh's Brownian motion and present  its relation with Walsh's Brownian motion. Then the stiff problem related to Walsh's Brownian motion will be described and we shall build a phase transition for it. The snapping out Walsh's Brownian motion corresponds to the so-called semi-permeable pattern of this stiff problem.
\end{abstract}


\tableofcontents

\section{Introduction}

In a previous work \cite{LS18}, we studied the stiff problems in one-dimensional space by means of Dirichlet forms. Let us briefly explain these two terminologies `stiff problem' and `Dirichlet form'. The former one was used in \cite{SanchezPalencia:1980dg} for an interesting problem related to a thermal conduction model with a `singular' barrier. A barrier means a small domain, in which the given thermal conductivity is also very small. In short, it is mainly concerned with the behaviour of the flux (i.e. the solution to the heat equation) as the volume of the barrier decreases to $0$. The latter one is a symmetric closed form with Markovian property on an $L^2(E,m)$ space, where $E$ is a nice topological space and $m$ is a fully supported Radon measure on it. Dirichlet forms are closely related to symmetric Markov processes under the regular condition due to a series of important works by M. Fukushima and some others in 1970's. We refer more details to \cite{CF12, FOT11}. 

As we know, thermal conduction models also link Markov processes (more exactly, diffusion processes) very closely. A recent study \cite{L16} by A. Lejay took the probabilistic description of a stiff problem into account. It explored a special one-dimensional case such that the material has constant thermal conductivity out of the barrier, and we called it the Brownian case of stiff problem in \cite{LS18}. The main purpose of \cite{L16} was to introduce the so-called snapping out Brownian motion (SNOB in abbreviation) and to link it with the limit of the flux as the length of the barrier tends to $0$.

In \cite{LS18}, we reconsidered the general stiff problems in one-dimensional space in a  different way. It was highlighted that the thermal resistance, rather than the thermal conductivity, plays the essential role in these problems. At a heuristic level, the thermal resistance corresponds to the scale function of one-dimensional diffusion related to the thermal conduction. Scale function, with so-called speed measure and so-called killing measure together, characterizes a `nice' diffusion on $\mathbb{R}$ or an interval completely. It is a continuous and strictly increasing function. The main result of \cite{LS18} showed that the behaviour of the flow depends on the total thermal resistance of the barrier, which is defined roughly as follows. Let $I_\varepsilon:=(-\varepsilon,\varepsilon)$ be the barrier and $\lambda_\varepsilon$ the measure induced by the scale function related to the thermal conduction model with this barrier. Then $\bar{\gamma}:=\lim_{\varepsilon\rightarrow 0}\lambda_\varepsilon(I_\varepsilon)$ is called the total thermal resistance of the barrier. More precisely, we built a phase transition for the stiff problems in terms of $\bar{\gamma}$: If $\bar{\gamma}=\infty$, then the flow cannot cross the barrier; if $0<\bar{\gamma}<\infty$, then the flow can penetrate the barrier partially and if $\bar{\gamma}=0$, then the barrier makes no sense. Except for the last case, the flux becomes discontinuous at the barrier. Needless to say, the most interesting case is the semi-permeable one, i.e. $0<\bar{\gamma}<\infty$. In fact, the Brownian case in \cite{L16} is actually a special semi-permeable case, and the snapping out Brownian motion  
manifests partial penetrations and partial reflections (at the barrier) in its own. We also generalized the SNOB to the so-called snapping out Markov processes in \cite{LS18} for the probabilistic counterparts of semi-permeable cases in general stiff problems. 

In this paper, we shall move on to the study of a speical stiff problem in multi-dimensional space. We find that the approach centred upon snapping out Markov processes is also suitable for a model related to Walsh's Brownian motion (WBM in abbrevation) on $\mathbb{R}^2$. The Walsh's Brownian motion, raised by J. Walsh in \cite{W78}, is a diffusion process on $\mathbb{R}^2$. But it is more like a one-dimensional diffusion, since it walks on each ray starting from the origin separately if the origin is ruled out. The origin is a switching point, which pieces together all the separate parts on the rays. Thus we could take the following stiff problem into account. Let $B(0,\varepsilon):=\{x\in \mathbb{R}^2: |x|<\varepsilon\}$ and we make the thermal conductivity  become very small in $B(0,\varepsilon)$. Intuitively speaking, a new diffusion process is obtained by attaching a small barrier (at $B(0,\varepsilon)$) to WBM. Then what matters in the stiff problem related to WBM is the limit of this diffusion process as $\varepsilon\rightarrow 0$. 

Two main results will be presented in this article. The first one is to introduce the so-called snapping out Walsh's Brownian motion and to derive its Dirichlet form in Theorem~\ref{PRO314}. This is a direct application of \cite[\S3]{LS18}. We also present an interesting link between the snapping out Walsh's Brownian motion and Walsh's Brownian motion by involving two kinds of transforms: one is the time-change and the other is the darning transform. The second result builds a phase transition of the stiff problem related to WBM in Theorem~\ref{THM51}. As an analogue of that in \cite[\S4]{LS18}, this phase transition is also based on the total thermal resistance of the barrier. Moreover, the continuity of phase transition is further derived in Theorem~\ref{THM43}. Note that the snapping out Walsh's Brownian motion actually corresponds to the semi-permeable pattern of this model.

\section{A review of snapping out Markov processes}\label{SNMP}

In \cite{LS18}, the so-called snapping out Markov process was introduced to describe the semi-permeable pattern of thermal conduction. For readers' convenience, we review its definition and main concerns in this short section.

Let $E$ be a locally compact separable metric space and $m$ a positive Radon measure fully supported on $E$. Further let $(\EE,\FF)$ be a regular Dirichlet form on $L^2(E,m)$ associated with a Markov process $X=(X_t)_{t\geq 0}$. Each function in a Dirichlet space is taken to be a quasi-continuous version. 
Take a positive, finite smooth measure $\mu$ on $E$. Then the \emph{snapping out Markov process}, denoted by $X^\s$, brings into play two transforms as follows:
\begin{itemize}
\item[(1)] Killing transform induced by $\mu$ on $X$: We refer this transform to \cite[\S6.1]{FOT11} (or \cite[\S2]{LS18}). The subprocess after killing is denoted by $X^\mu=(X^\mu_t)_{t\geq 0}$. Note that the Dirichlet form of $X^\mu$ is the so-called perturbed Dirichlet form:
\[
	\begin{aligned}
  \FF^\mu &=\FF\cap L^2(E,\mu),\\
 \EE^\mu(f,g)&=\EE(f,g)+\int_E fgd\mu,\quad f,g\in \FF^\mu.
 \end{aligned}
\]
\item[(2)] Piecing out transform with instantaneous distribution $\mu^{\#}:=\mu/\mu(E)$ on $X^\mu$: This transform was introduced in \cite{IM74}, and also reviewed in \cite[\S2]{LS18}. We only give a brief explanation at a heuristic level here. Loosely speaking,  it pieces together a new trajectory starting from a reborn site, which is chosen randomly according to $\mu^\#$, once the trajectory of $X^\mu$ dies. In some sense, piecing out transform could be treated as the converse of killing.
\end{itemize}
In other words, $X^\s$ is, by definition, a Markov process obtained by applying the piecing out transform on $X^\mu$. It was shown in \cite{LS18} that if $X$ has no killing inside, then $X^\s$ is $m$-symmetric. Furthermore, the Dirichlet form of $X^\s$ is 
\begin{equation}\label{EQ3FSU}
 \begin{aligned}
  \FF^\mathrm{s}&=\left\{u\in \FF: \int_{E\times E}\left(u(x)-u(y)\right)^2\mu(dx)\mu(dy)<\infty\right\}, \\
 \EE^\mathrm{s}(u,v)&=\EE(u,v)+\frac{1}{2|\mu|}\int_{E\times E}\left(u(x)-u(y)\right)\left(v(x)-v(y)\right)\mu(dx)\mu(dy),\quad u,v\in \FF^\s,
 \end{aligned}
\end{equation}
where $|\mu|:=\mu(E)$. 


In \cite{LS18}, only snapping out Markov processes on $\mathbb{G}:=(-\infty, 0-]\cup [0+,\infty)$, in which $0+$ and $0-$ formally identified with $0\in \mathbb{R}$ are two isolated points, are paid attentions. They are the probabilistic counterparts of semi-permeable patterns appeared in the phase transitions of stiff problems in one-dimensional space (Cf. \cite[Theorem~4.6]{LS18}). Instead, we shall formulate the so-called snapping out Walsh's Brownian motion on a two-dimensional space in next section. Needless to say, it is also motivated by a related stiff problem. 

\section{Snapping out Walsh's Brownian motion}\label{SEC33}

This section is devoted to the studies of snapping out Walsh's Brownian motion (SNOWB in abbreviation) and its connections with Walsh's Brownian motion (WBM in abbreviation).

\subsection{Walsh's Brownian motion}

Following \cite{CF15}, write $\mathbb{R}^2=\cup_{\theta\in [0,2\pi)}R_\theta$, where $R_\theta$ is a ray starting from the origin $\bf{0}$ with the angle $\theta$. Let $\eta$ be a fully supported probability measure on $S^1:=[0,2\pi)$ and $m(dx):=dr\eta(d\theta)$, where $x=(r,\theta)$ is the polar coordinate of $x\in \mathbb{R}^2$. 
The Walsh's Brownian motion introduced by Walsh \cite{W78} is a diffusion process on $\mathbb{R}^2$. It behaves like a one-dimensional Brownian motion on each $R_\theta$ away from the origin, and in case of hitting the origin, it may choose a new direction $\theta'$ according to $\eta$ and go on walking like a Brownian motion on $R_{\theta'}$ until it hits the origin again. More probabilistic descriptions of WBM are referred to \cite{BPY89, W78}. 

Instead, we note that Chen et al. reconstructed the WBM by means of Dirichlet form in \cite{CF15}. Precisely speaking, it is associated with a regular Dirichlet form $(\EE^W,\FF^W)$ on $L^2(\mathbb{R}^2,m)$ as follows
\begin{equation}\label{EQ3FWF}
\begin{aligned}
 & \begin{aligned}
  \FF^W=\bigg\{ f
  & \in L^2(\mathbb{R}^2,m): 
 f_{\theta}\in H^1\left((0,\infty)\right)~
\mbox{and}~\lim_{r\rightarrow0}f_{\theta}(r)=c ~\mbox{for}~\eta\mbox{-a.e.}~\theta\\
&\qquad\qquad\quad  \mbox{and some}~c~\mbox{independent of}~\theta, \int_{S^1}\mathbf{D}(f_{\theta},f_{\theta})\eta(d\theta)<\infty\bigg\}, 
\end{aligned}
\\
 & \EE^W(f,g)=\frac{1}{2}\int_{S^1}\mathbf{D}(f_{\theta},g_{\theta})\eta(d\theta),\quad f,g\in \FF^W,
 \end{aligned}
\end{equation}
where 
\[
\mathbf{D}(f_\theta,g_\theta):=\int_0^{\infty}f'_\theta(r)g'_\theta(r)dr,
\] 
and $f_\theta(r):=f(r,\theta)$ for any function $f$ on $\mathbb{R}^2=[0,\infty)\times S^1$. 
Its extended Dirichlet space is 
\begin{equation}\label{EQ3FWE}
\begin{aligned}
  \FF^W_\e=\bigg\{ f: 
 f_{\theta}\in & \text{BL}((0,\infty))~\mbox{and}~\lim_{r\rightarrow0}f_{\theta}(r)=c ~\mbox{for}~\eta\mbox{-a.e.}~\theta\\
& \mbox{and some}~c~\mbox{independent of}~\theta, \int_{S^1}\mathbf{D}(f_{\theta},f_{\theta})\eta(d\theta)<\infty\bigg\}, 
\end{aligned}
\end{equation}
where 
\[\text{BL}((0,\infty))=\{h:h~\mbox{is absolutely continuous on}~(0,\infty), \mathbf{D}(h,h)<\infty\}.\]
Particularly, $(\EE^W,\FF^W)$ is irreducible and recurrent. 

\subsection{Snapping out WBM}

Let us turn our attentions to the SNOWB. The state space of SNOWB is $\mathbb{G}^2$ obtained by viewing the origin $\bf{0}$ of $\mathbb{R}^2$ as a circle, which is homeomorphic to $S^1$. In other words, 
\[
	\mathbb{G}^2:=[0,\infty)\times S^1,
\]
but $(0,\theta_1), (0,\theta_2)$ are distinct points if $\theta_1\neq\theta_2$. Equivalently, $\mathbb{G}^2$ is topologically homeomorphic to $[1,\infty)\times S^1$ via the transform
\[
	\mathbf{T}_{1}: \mathbb{G}^2\rightarrow [1,\infty)\times S^1,\quad (r, \theta)\mapsto (r+1, \theta). 
\]
We should write 
\[
	\mathbb{G}^2=\mathbb{G}_+\times S^1=[0+,\infty)\times S^1
\]
if no confusion caused.  Clearly, $m(dx)=dr\eta(d\theta)$ is a fully supported Radon measure on $\mathbb{G}^2$. 

To introduce the SNOWB, we start with a reflecting WBM on $\mathbb{G}^2$, which is a union of separate reflecting Brownian motions on the rays of $\mathbb{G}^2$.  It is not irreducible and given by the Dirichlet form in the following lemma. 

\begin{lemma}\label{LM31}
The quadratic form 
\begin{equation}\label{EQ4FFL}
 \begin{aligned}
 & \FF=\left\{f\in L^2(\mathbb{G}^2,m): 
  f_{\theta}\in H^1\left([0+,\infty)\right) ~\mbox{for}~\eta\mbox{-a.e.}~\theta,~\mbox{and }
\int_{S^1}\mathbf{D}(f_{\theta},f_{\theta})\eta(d\theta)<\infty\right\},
\\
 & \EE(f,g)=\frac{1}{2}\int_{S^1}\mathbf{D}(f_{\theta},g_{\theta})\eta(d\theta),\quad f,g\in \FF,
 \end{aligned}
\end{equation}
is a regular Dirichlet form on $L^2(\mathbb{G}^2,m)$. The extended Dirichlet space of $(\EE,\FF)$ is 
\begin{equation}\label{EQ3FEFF}
  \FF_\e=\bigg\{ f: 
 f_{\theta}\in \text{BL}([0+,\infty))~\mbox{for}~\eta\mbox{-a.e.}~\theta,\text{ and } \int_{S^1}\mathbf{D}(f_{\theta},f_{\theta})\eta(d\theta)<\infty\bigg\}, 
\end{equation}
where 
\[\text{BL}([0+,\infty))=\{h:h~\mbox{is absolutely continuous on}~[0+,\infty), \mathbf{D}(h,h)<\infty\}.
\]
Particularly, $(\EE,\FF)$ is recurrent. 
\end{lemma}
\begin{proof}
Clearly, $(\EE,\FF)$ is a symmetric bilinear form with Markovian property. It suffices to prove the closeness and regularity. The idea of these proofs is due to \cite{CF15}. 
To prove the closeness of \eqref{EQ4FFL}, let  $\{u_n\}$ be an $\EE_1$-Cauchy sequence in $\FF$. By taking a subsequence if necessary, we may assume $\EE_1(u_{n+1}-u_n,u_{n+1}-u_n)<2^{-n}$, so that
\[
\sum_{n=1}^{\infty}\EE_1(u_{n+1}-u_n,u_{n+1}-u_n)<\infty.
\]
By Cauchy-Schwarz inequality and Fubini theorem, we have
\[
\int_{S_1}\sum_{n=1}^{\infty}\left(\int_0^{\infty}\left(u_{n+1,\theta}(r)-(u_{n,\theta}(r)\right)^2+\left(u'_{n+1,\theta}(r)-(u'_{n,\theta}(r)\right)^2dr\right)^{\frac{1}{2}}\eta(d\theta)<\infty,
\]
where $u_{n,\theta}(r):=u_n(r,\theta)$. This implies that there exists a set $A\subset S^1$ with $\eta(A)=0$ such that for every $\theta\in S^1\setminus A$, (recall that $\mathbb{G}_+=[0+,\infty)$)
\[
\sum_{n=1}^{\infty}\|u_{n+1,\theta}-u_{n,\theta}\|_{H^1(\mathbb{G}_+)}<\infty.
\]
Thus a function $u_{\theta}\in H^1(\mathbb{G}_+)$ exists for any $\theta\in S^1\setminus A$ such that $\|u_{n,\theta}-u_{\theta}\|_{H^1(\mathbb{G}_+)}\rightarrow 0$ as $n\rightarrow 0$. It follows from Fatou lemma that
\[\int_{S^1}\|u_{\theta}\|_{H^1(\mathbb{G}_+)}\eta(d\theta)<\infty,\quad \int_{S^1}\|u_{n,\theta}-u_{\theta}\|_{H^1(\mathbb{G}_+)}\eta(d\theta)\rightarrow 0.
\]
Consequently, $u(r,\theta):=u_\theta(r)\in \FF$ and $u_n$ is $\EE_1$-convergent to $u$. In other words, $(\EE,\FF)$ is a closed form. 
Note that $C^{\infty}_{c}(\mathbb{G}^2)\subset\FF$, and it suffices to prove that $C^{\infty}_{c}(\mathbb{G}^2)$ is $\EE_1$-dense in $\FF$ for the regularity. In fact, suppose $f\in\FF$ such that $\EE_1(f,g)=0$ for every $g\in C^{\infty}_{c}(\mathbb{G}^2)$. Take $g(r,\theta)=\phi(r)\psi({\theta})$ with $\phi\in C^{\infty}_c\left(\mathbb{G}_+\right)$ and $\psi\in C^{\infty}(S^1)$, and we have
\[
\int_{S^1}\mathbf{D}_1(f_{\theta},\phi)\psi(\theta)\eta(d\theta)=0.
\]
It follows that there exist a family of countable functions $\{\phi_n\}\subset C^{\infty}_c\left(\mathbb{G}_+\right)$ dense in $H^1(\mathbb{G}_+)$ and a set $A\subset S^1$ with $\eta(A)=0$ such that $\mathbf{D}_1(f_\theta,\phi_n)=0$ for any $\theta\in S^1\setminus A$. This implies $f_{\theta}=0,\theta\in S^1\setminus A$ and thus $f=0$, $m$-a.e. on $\mathbb{G}^2$.

Let us turn to prove \eqref{EQ3FEFF}. Denote the right side of \eqref{EQ3FEFF} by $\mathcal{G}$. The family of all bounded functions in $\mathcal{G}$ is denoted by $\mathcal{G}_b$.  Let $f\in \FF_\e$ and $\{f_n\}\subset \FF$ be its approximation sequence. Mimicking the proof of closeness, we can conclude that $\{f'_{n,\theta}:n\geq 1\}$ is $L^2(\mathbb{G}_+)$-Cauchy and $f_{n,\theta}(r_\theta)\rightarrow f_\theta(r_\theta)$ with some $r_\theta\in \mathbb{G}_+$ for any $\theta\in S^1\setminus A$ with $\eta(A)=0$. Thus $f'_{n,\theta}\rightarrow g_\theta$ in $L^2(\mathbb{G}_+)$ for some $g_\theta \in L^2(\mathbb{G}_+)$. Set
\[
	\hat{f}(r,\theta):=f_\theta(r_\theta)+\int_{r_\theta}^r g_\theta(u)du,\quad r\in \mathbb{G}_+, \theta\in S^1\setminus A. 
\]
We can easily deduce that $f_n(r,\theta)\rightarrow \hat{f}(r,\theta)$ for $r\in \mathbb{G}_+$ and $\theta\in S^1\setminus A$. Thus $f=\hat{f}$, $m$-a.e.  
Clearly, $\hat{f}\in \mathcal{G}$ and we have $\FF_\e\subset \mathcal{G}$. To the contrary,  we need only prove $\mathcal{G}_b\subset \FF_\e$ by \cite[Lemma~1.1.12]{CF12}. Let $f\in \mathcal{G}_b$ with $\|f\|_\infty<M$ for some $M>0$. Further assume $f_\theta \in \text{BL}([0+,\infty))$ for $\theta\in S^1\setminus A$ with $\eta(A)=0$. Take for each integer $n$ a smooth function $\varphi_n\in C_c^\infty([0+,\infty))$ such that 
\[
	0\leq \varphi_n\leq 1, \quad \varphi_n|_{[0+, n]}=1,\quad \varphi_n|_{[2n+1,\infty)}=0,\quad |\varphi'_n|\leq 1/n.
\]
Set 
\[
f_n(r,\theta):= f(r,\theta)\cdot \varphi_n(r),\quad r\in \mathbb{G}_+, \theta\in S^1\setminus A. 
\]
Clearly, $f_n\in \FF$ and we can also deduce that $\{f_n\}$ is an approximation sequence of $f$. Therefore, $f\in \FF_\e$. 

The recurrence of $(\EE,\FF)$ is implied by the fact that $1\in \FF_\e, \EE(1,1)=0$. 
That completes the proof.
\end{proof}
\begin{remark}
The Dirichlet form $(\EE,\FF)$ in this lemma is not irreducible. Indeed, any set $[0+,\infty)\times A$ with $A\subset S^1$ being Borel measurable is an invariant set of $(\EE,\FF)$. 
\end{remark}

The SNOWB with a parameter $\kappa>0$ is, by definition, the snapping out Markov process with respect to $(\EE,\FF)$ and a finite measure $\mu(drd\theta):=\kappa \delta_0(dr)\eta(d\theta)$, where $\delta_0$ is the Dirac measure centered on $0+\in \mathbb{G}_+$ hereafter. The smoothness of $\mu$ is implied by the following lemma.

\begin{lemma}\label{LM313}
Let $(\EE,\FF)$ be given by \eqref{EQ4FFL} and $\kappa>0$. Then $\mu(drd\theta)=\kappa \delta_0(dr)\eta(d\theta)$ is smooth with respect to $(\EE,\FF)$.
\end{lemma}

\begin{proof}
We assert $\mu$ is of finite energy integral, which implies that $\mu$ is smooth. Indeed, for any $v\in\FF\cap C_c(\mathbb{G}^2)$, $v_{\theta}\in H^1(\mathbb{G}_+)$ for $\eta$-a.e. $\theta$. Thus
\[|v_{\theta}(0)| \leq C\|v_{\theta}\|_{H^1(\mathbb{G}_+)}
\]
for some constant $C>0$ independent of $v$ and 
\[
\int |v|d\mu= \kappa \int_{S^1}|v(0, \theta)|\eta(d\theta)\leq C\kappa\int_{S^1}\|v_{\theta}\|_{H^1(\mathbb{G}_+)}\eta(d\theta)\leq 2C\kappa \sqrt{\EE_1(v,v)}.
\]
That completes the proof. 
\end{proof}

Then we can conclude the following assertions about the SNOWB. 

\begin{theorem}\label{PRO314}
Let $(\EE,\FF)$ be given by \eqref{EQ4FFL}. Then the SNOWB is associated with a regular Dirichlet form $(\EE^\s, \FF^\s)$ on $L^2(\mathbb{G}^2, m)$ as follows:
\begin{equation}\label{EQ3FSFF}
\begin{aligned}
 \FF^\s &= \FF, \\
\EE^\s(f,g)&=
  \frac{1}{2}\int_{S^1}\mathbf{D}(f_{\theta},g_{\theta})\eta(d\theta)\\
 &+\dfrac{\kappa}{2}\int_{S^1\times S^1}\left(f(0, \theta)-f(0,\theta')\right)\left(g(0, \theta)-g(0, \theta')\right)\eta(d\theta)\eta(d\theta'),\quad f,g\in \FF^\s.
 \end{aligned}
\end{equation}
The extended Dirichlet space $\FF^\s_\e$ of $(\EE^\s,\FF^\s)$ is
\[
\begin{aligned}
  \FF^\s_\e=\bigg\{ f: 
 f_{\theta}\in \text{BL}([0+,\infty))~&\mbox{for}~\eta\mbox{-a.e.}~\theta, \int_{S^1}\mathbf{D}(f_{\theta},f_{\theta})\eta(d\theta)<\infty,\\ & \int_{S^1\times S^1}\left(f(0, \theta)-f(0, \theta')\right)^2\eta(d\theta)\eta(d\theta')<\infty\bigg\}. 
\end{aligned}\]
Particularly, $(\EE^\s,\FF^\s)$ is irreducible and recurrent. 
\end{theorem}
\begin{proof}
Note that for any $f\in \FF$, 
\begin{equation*}
\begin{aligned}
\left(f(0, \theta)-f(0, \theta')\right)^2& \leq 2(f(0,\theta)^2+f(0,\theta')^2) \\
&\leq C \left(\mathbf{D}_1(f_\theta, f_\theta)+\mathbf{D}_1(f_{\theta'}, f_{\theta'}) \right)
\end{aligned}
\end{equation*}
with some constant $C$ independent of $f$. Thus 
\[
\int_{S^1\times S^1}\left(f(0, \theta)-f(0, \theta')\right)^2\eta(d\theta)\eta(d\theta')\leq 4C\EE_1(f,f). 
\]
This indicates \eqref{EQ3FSFF} by \eqref{EQ3FSU}. The expression of $\FF^\s_\e$ is implied by  \cite[Proposition~3.8]{LS18} and \eqref{EQ3FEFF}. Since $(\EE,\FF)$ is recurrent by Lemma~\ref{LM31}, it follows from \cite[Proposition~3.8~(1)]{LS18} that $(\EE^\s,\FF^\s)$ is also recurrent. 

To show the irreducibility of $(\EE^\s, \FF^\s)$, let $f\in \FF^\s$ with $\EE^\s(f,f)=0$. This means
\[
\int_{S^1}\mathbf{D}(f_\theta,f_\theta)\eta(d\theta)=\int_{S^1\times S^1}\left(f(0, \theta)-f(0,\theta')\right)^2\eta(d\theta)\eta(d\theta')=0. 
\]
As a consequence, there exists a set $A\subset S^1$ with $\eta(A)=0$ such that for any $\theta, \theta'\in S^1\setminus A$, 
\[
f_\theta\in H^1(\mathbb{G}_+),\quad \mathbf{D}(f_\theta,f_\theta)=0,\quad f(0, \theta)-f(0,\theta')=0.
\]
Clearly, $\mathbf{D}(f_\theta,f_\theta)=0$ indicates $f_\theta$ is a constant function. Then it follows from $f(0, \theta)-f(0,\theta')=0$ that $f$ is constant $m$-a.e. Therefore, we obtain the irreducibility of $(\EE^\s, \FF^\s)$ from \cite[Theorem~5.2.16]{CF12}. That completes the proof. 
\end{proof}

\begin{remark}
It does not always hold that $\FF_\e=\FF^\s_\e$. For example, assume $\eta$ is the uniform distribution on $S^1$ and take 
\[
	f(r,\theta)=\theta^{-1},\quad r\in [0+,\infty), \theta\in (0,2\pi). 
\]
Clearly, $f\in \FF_\e$, while $\int_{S^1\times S^1}\left(f(0, \theta)-f(0, \theta')\right)^2d\theta d\theta'$ diverges. 
\end{remark}

\subsection{Links between WBM and SNOWB}

Now we pursue to study the links between WBM and SNOWB. Two transforms of Markov processes are involved in the principal result. One is the time-change, whose counterpart in the theory of Dirichlet form is the so-called trace Dirichlet form. This transform is very well known, and we hesitate to repeat its details for the sake of brevity. Instead, we refer to \cite[Chapter 5]{CF12} as well as \cite[\S2]{LS18}. The other is the darning transform, which was raised in \cite{CF08}. Roughly speaking, it collapses a compact subset of the state space to an abstract point and produces a new Markov process by `erasing' the information contained in this compact set. Particularly, given a regular Dirichlet form $(\EE,\FF)$ on $L^2(E,m)$ and a compact set $K\subset E$ of postive capacity, the Markov process obtained by the darning transform, which shorts $K$ into $a^*$, is given by the Dirichlet form on $L^2(E^*,m^*)$
\begin{equation}\label{DR}
 \begin{aligned}
 & \FF^*=\left\{f^*: f\in \FF, f \mbox{ is constant $\EE$-q.e. on $K$}\right\},\\
 & \EE^*(f^*,g^*)=\EE(f,g),\quad f^*,g^*\in \FF^*,
 \end{aligned}
\end{equation}
where $E^*:=(E\setminus K)\cup \{a^*\}$, $m^*|_{E\setminus K}:=m$, $m^*(\{a^*\})=0$ and $f^*|_{E\setminus K}:=f$, $f^*(a^*):=f(x)$ with some $x\in K$. 
We refer a relevant study of darning transform to \cite{CP17}. A short review is also presented in \cite[\S2]{LS18}. 

The notation
\[
	\mathbf{T}_{\beta}: \mathbb{G}^2\rightarrow \{x\in \mathbb{R}^2: |x|\geq \beta\},\quad (r, \theta)\mapsto (r+\beta, \theta)
\]
with $\beta>0$ stands for the homeomorphism between $\mathbb{G}^2$ and $\{x\in \mathbb{R}^2: |x|\geq \beta\}$. The following theorem is inspired by an analogical result \cite[Theorem~3.12]{LS18}, which links the snapping out Brownian motion with one-dimensional Brownian motion. It tells us after a spatial transform the SNOWB is the trace of WBM on a certain closed set, and on the contrary, the WBM is the darning of SNOWB by shorting $\{0+\}\times S^1$ into $\bf{0}$. 

\begin{theorem}\label{THM317}
\begin{itemize}
\item[(1)] By shorting $\{0+\}\times S^1$ into $\bf{0}$, the Markov process with darning induced by the SNOWB is the Walsh's Brownian motion. 
\item[(2)] Denote the SNOWB by $W^\s=(W^\s_t)_{t\geq 0}$. Let $F_\kappa:=\{x\in \mathbb{R}^2: |x|\geq (2\kappa)^{-1}\}$ and $m_\kappa(drd\theta):= dr\eta(d\theta)$ on $F_\kappa$. Then $\mathbf{T}_{(2\kappa)^{-1}}(W^\s)$ is a Markov process on $F_\kappa$ associated with the trace Dirichlet form of $(\EE^W,\FF^W)$ on $F_\kappa$ with the speed measure $m_\kappa$. 
\end{itemize}
\end{theorem}
\begin{proof}
The first assertion is clear by applying \eqref{DR} with $K:=\{0+\}\times S^1$ and $a^*:=\mathbf{0}$ to \eqref{EQ3FSFF}. For the second assertion, it suffices to characterize the trace Dirichlet form $(\check{\EE},\check{\FF})$ of $(\EE^W,\FF^W)$ on $F_\kappa$. Write $a:=(2\kappa)^{-1}$, $F:=F_\kappa$, $\partial F:=\{x\in \mathbb{R}^2: |x|=a\}$ and $\mathfrak{m}:=m_\kappa$ (on $F_\kappa$) for convenience. 
Recall that for appropriate function $f$ on $F$, $\mathbf{H}_F$ denotes the hitting distribution of $W$ for $F$, i.e. 
\[\mathbf{H}_Ff(x)=\mathbf{E}_x[f(W_{\sigma_F}), \sigma_F<\infty],   
\]
and $\sigma_F$ denotes the hitting time of $F$ with respect to $W$. For $\lambda>0$, we also write
\[
\mathbf{H}^{\lambda}_Ff(x)=\mathbf{E}_x[e^{-\lambda\sigma_F}f(W_{\sigma_F}), \sigma_F<\infty]. 
\]

A first step towards the trace Dirichlet form is to prove the following assertion: For any non-negative bounded function $\phi$ on $F$ and $x=(r,\theta)\in G:=F^c$ (i.e. $r<a, \theta\in S^1$),
\begin{align}
& \mathbf{H}_F\phi(r,\theta)=\dfrac{r}{a}\phi(a,\theta)+\left(1-\dfrac{r}{a}\right)\bar{\phi}(a),\label{HD}\\
&\mathbf{H}^{\lambda}_F\phi(r,\theta)=\dfrac{\sinh(\sqrt{2\lambda}r)}{\sinh(\sqrt{2\lambda}a)}\phi(a,\theta)+\dfrac{\sinh(\sqrt{2\lambda}(a-r))}{\sinh(\sqrt{2\lambda}a)\cosh(\sqrt{2\lambda}a)}\bar{\phi}(a),\label{HDA}
\end{align}
where $\bar{\phi}(a):=\int_{S^1}\phi(a,\theta)\eta(d\theta)$. 

We first consider the case $0<r<a$. The continuity of $W$ implies $W_{\sigma_F}\in \partial F$, $\mathbf{P}_x$-a.s. for $x\in G$. Denote the hitting time of $\{\bf{0}\}$ with respect to $W$ by $\sigma_{\bf{0}}$. We have
\[
\mathbf{H}_F\phi(r,\theta) =\mathbf{E}_{(r,\theta)}\left[\phi(W_{\sigma_F})1_{\{\sigma_F<\sigma_{\bf{0}}\}}\right]+\mathbf{E}_{(r,\theta)}\left[\phi(W_{\sigma_F})1_{\{\sigma_F>\sigma_{\bf{0}}\}}\right].
\]
Note that $\sigma_F<\sigma_{\bf{0}}$ amounts to $\sigma_{(a,\theta)}<\sigma_{\bf{0}}$, where $\sigma_{(a,\theta)}$ is the hitting time of $\{(a,\theta)\}$ with respect to $W$. In the meantime, $W_{\sigma_F}=(a,\theta)$ and thus
\begin{equation}\label{EQ3ERT}
\mathbf{E}_{(r,\theta)}\left[\phi(W_{\sigma_F})1_{\{\sigma_F<\sigma_{\bf{0}}\}}\right]=\phi(a,\theta)\mathbf{P}_{(r,\theta)}[\sigma_{(a,\theta)}<\sigma_{\bf{0}}].
\end{equation}
Let $B=(B_t)_{t\geq 0}$ be a one-dimensional Brownian motion and $\tau_x$ be the hitting time of $\{x\}$ with respect to $B$ for $x\in \mathbb{R}$. Note that $|W|$ is a reflecting Brownian motion on $[0,\infty)$ (Cf. \cite{W78}) and thus has the same distribution as $|B|$. It follows that
\begin{equation}\label{EQ3PRT}
\mathbf{P}_{(r,\theta)}[\sigma_{(a,\theta)}<\sigma_{\bf{0}}]=
\mathbf{P}^B_r[\tau_a<\tau_0]=\frac{r}{a},
\end{equation}
where $\mathbf{P}^B_r$ is the probability measure of $B$ starting from $r$. 
The last equality follows from Problem 6 of \S 1.7 in \cite{IM74}. When $\sigma_F>\sigma_{\bf{0}}$, we can deduce from the strong Markov property of $W$ that
\[
\mathbf{E}_{(r,\theta)}\left[\phi(W_{\sigma_F})1_{\{\sigma_F>\sigma_{\bf{0}}\}}\right]=\mathbf{E}_{(r,\theta)}\left[1_{\{\sigma_F>\sigma_{\bf{0}}\}}\mathbf{E}_{\bf{0}}[\phi(W_{\sigma_F})]\right].
\]
Thanks to \cite{BPY89, W78}, we know that $W_t=(|W_t|, \Psi_t)$ in polar coordinate system is such that $|W|$ is independent of $\Psi$, and $\Psi_t$ is distributed as $\eta$ for any $t$ under $\mathbf{P}_{\bf{0}}$. Particularly, $\sigma_F=\inf\{t>0: |W_t|\geq a\}$ is independent of $\Psi$ under $\mathbf{P}_{\bf{0}}$. Hence 
\begin{equation}\label{EQ3EWF}
	\mathbf{E}_{\bf{0}}[\phi(W_{\sigma_F})]=\mathbf{E}_{\bf{0}}[\phi(a, \Psi_{\sigma_F})]=\bar{\phi}(a)
\end{equation}
and we have
\begin{equation}\label{EQ3ERTP}
\mathbf{E}_{(r,\theta)}\left[\phi(W_{\sigma_F})1_{\{\sigma_F>\sigma_{\bf{0}}\}}\right]=\left(1-\frac{r}{a}\right)\bar{\phi}(a).
\end{equation}
Then \eqref{HD} follows from \eqref{EQ3ERT}, \eqref{EQ3PRT} and \eqref{EQ3ERTP}. To prove \eqref{HDA}, we have
\begin{equation}\label{EQ3HFL}
\mathbf{H}_F^{\lambda}\phi(r,\theta)
 =\mathbf{E}_{(r,\theta)}\left[e^{-\lambda\sigma_F}\phi(W_{\sigma_F})1_{\{\sigma_F<\sigma_{\bf{0}}\}}\right]+\mathbf{E}_{(r,\theta)}\left[e^{-\lambda\sigma_F}\phi(W_{\sigma_F})1_{\{\sigma_F>\sigma_{\bf{0}}\}}\right].
\end{equation}
The first term of \eqref{EQ3HFL} equals
\begin{equation}\label{EQ3ERTE}
\mathbf{E}_{(r,\theta)}\left[e^{-\lambda\sigma_F}\phi(W_{\sigma_F})1_{\{\sigma_F<\sigma_{\bf{0}}\}}\right]=
\phi(r,\theta)\mathbf{E}^B_r\left[e^{-\lambda\tau_a};\tau_a<\tau_0\right]
\end{equation}
and by the strong Markov property of $W$, the second term of \eqref{EQ3HFL} equals
\[
\mathbf{E}_{(r,\theta)}\left[e^{-\lambda\sigma_F}\phi(W_{\sigma_F})1_{\{\sigma_F>\sigma_{\bf{0}}\}}\right]
=\mathbf{E}_{(r,\theta)}\left[e^{-\lambda\sigma_{\bf{0}}}1_{\{\sigma_F>\sigma_{\bf{0}}\}}\mathbf{E}_{\bf{0}}\left[e^{-\lambda\sigma_F}\phi(W_{\sigma_F})\right]\right]. 
\]
Similar to \eqref{EQ3EWF}, we obtain
\begin{equation}\label{EQ3EEL}
\mathbf{E}_{\bf{0}}\left[e^{-\lambda\sigma_F}\phi(W_{\sigma_F})\right]=\mathbf{E}_{\bf{0}}\left[e^{-\lambda\sigma_F}\right]\bar{\phi}(a).
\end{equation}
Since $|W|$ has the same distribution as $|B|$, it follows that
\begin{equation}\label{EQ3ERTEL}
\mathbf{E}_{(r,\theta)}\left[e^{-\lambda\sigma_F}\phi(W_{\sigma_F})1_{\{\sigma_F>\sigma_{\bf{0}}\}}\right]
=\mathbf{E}^B_r\left[e^{-\lambda\tau_0};\tau_0<\tau_a\right]\mathbf{E}_0^B\left[e^{-\lambda(\tau_{-a}\wedge\tau_a)}\right]\bar{\phi}(a).
\end{equation}
Problem 6 of \S1.7 in \cite{IM74} implies
\begin{equation}\label{EQ3EBR}
\begin{aligned}
& \mathbf{E}^B_r\left[e^{-\lambda\tau_a};\tau_a<\tau_0\right]=\dfrac{\sinh(\sqrt{2\lambda}r)}{\sinh(\sqrt{2\lambda}a)},\\
& \mathbf{E}^B_r\left[e^{-\lambda\tau_0};\tau_0<\tau_a\right]=\dfrac{\sinh(\sqrt{2\lambda}(a-r))}{\sinh(\sqrt{2\lambda}a)},\\
& \mathbf{E}^B_0\left[e^{-\lambda(\tau_{-a}\wedge\tau_a)}\right]=\dfrac{1}{\cosh(\sqrt{2\lambda}a)}.
\end{aligned}
\end{equation}
Hence \eqref{HDA} follows from \eqref{EQ3ERTE} and \eqref{EQ3ERTEL}.

When $r=0$, \eqref{HD} is implied by \eqref{EQ3EWF}, and \eqref{HDA} follows from \eqref{EQ3EEL} and the last equality of \eqref{EQ3EBR}. 

Now we claim that the trace Dirichlet form $(\check{\EE},\check{\FF})$ on $L^2(F,\m)$ is given by 
\begin{equation*}
 \begin{aligned}
 & \check{\FF}=\left\{f\in L^2(F,\m): 
  f_{\theta}\in H^1([a,\infty)) ~\mbox{for}~\eta\mbox{-a.e.}~\theta,
\int_{S^1}\mathbf{D}^{(a)}(f_{\theta},f_{\theta})\eta(d\theta)<\infty\right\},
\\
 & 
 \begin{aligned}
 \check{\EE}(f,g)=
 &~\frac{1}{2}\int_{S^1}\mathbf{D}^{(a)}(f_{\theta},g_{\theta})\eta(d\theta)\\
 &~+\frac{1}{4a}\int_{S^1\times S^1}\left(f_{\theta}(a)-f_{\theta'}(a)\right)\left(g_{\theta}(a)-g_{\theta'}(a)\right)\eta(d\theta)\eta(d\theta'),\quad f,g\in \check{\FF},
 \end{aligned}
 \end{aligned} 
\end{equation*}
where $\mathbf{D}^{(a)}(f_\theta,g_\theta):=\int_a^{\infty}f'_\theta(r)g'_\theta(r)dr$. Indeed, the expression of $\check{\FF}$ follows from \cite[(2.2)]{LS18} and \eqref{EQ3FWE}.
Since $(\EE^W,\FF^W)$ is recurrent, it follows from \cite[Theorem 5.2.5 and Proposition 2.1.10]{CF12} that $(\check{\EE},\check{\FF})$ is conservative. Thus $(\check{\EE},\check{\FF})$ has no killing inside and from \cite[Corollary 5.6.1]{CF12} we can obtain that for any $f\in \check{\FF}_\e$, 
\begin{equation}\label{TR}
\check{\EE}(f,f)=\frac{1}{2}\mu_{\langle\mathbf{H}_Ff\rangle}(F)+\frac{1}{2}\int_{F\times F\setminus \mathtt{d}}\left(f(x)-f(y)\right)^2U(dx,dy),
\end{equation}
where $\mu_{\langle\mathbf{H}_Ff\rangle}$ is the energy measure of $(\EE^W,\FF^W)$ relative to $\mathbf{H}_Ff$, $\mathtt{d}$ is the diagonal of $F\times F$ and $U$ is the Feller measure of $W$ on $F\times F\setminus \mathtt{d}$. We refer the details of Feller measure to \cite{CF12, CFY06}. In what follows, we shall first compute the local term of \eqref{TR} and then formulate the Feller measure $U$.  In fact, for any $\varphi\in C^1_c(\mathbb{R}^2)$ and $u\in\FF^W_\e$, 
\[\int_{\mathbb{R}^2}\varphi d\mu_{\langle u\rangle}=2\EE^W(u\varphi,u)-\EE^W(u^2,\varphi)=\int_0^{\infty}\int_{S^1}u'_{\theta}(r)^2\varphi_{\theta}(r)\eta(d\theta)dr.\]
This implies
\[ d\mu_{\langle u\rangle}=u'_{\theta}(r)^2\eta(d\theta)dr.\]
Since $\mathbf{H}_Ff=f$ on $F$, it follows that
\begin{equation}\label{EM}
\mu_{\langle\mathbf{H}_Ff\rangle}(F)=\int_Ff'_{\theta}(r)^2\eta(d\theta)dr
=\int_{S^1}\mathbf{D}^{(a)}(f_{\theta},f_{\theta})\eta(d\theta).
\end{equation}
To formulate the Feller measure $U$, take two non-negative bounded functions $\phi$ and $\psi$ on $F$ such that $\phi\cdot\psi\equiv0$. From \cite[(5.5.13) and (5.5.14)]{CF12}, we know that
\begin{equation}\label{EQ3UPP}
U(\phi\otimes\psi)=\uparrow\lim_{\lambda\uparrow\infty}\lambda(\mathbf{H}^{\lambda}_F\phi, \mathbf{H}_F\psi)_{F^c}. 
\end{equation}
Substituting \eqref{HD} and \eqref{HDA} in \eqref{EQ3UPP}, we obtain
\[
\begin{split}
U(\phi\otimes\psi)
& =\lim_{\lambda\uparrow\infty}\lambda
\bar{\phi}(a)\bar{\psi}(a)\int_0^a\left(\dfrac{\sinh(\sqrt{2\lambda}(a-r))}{\sinh(\sqrt{2\lambda}a)\cosh(\sqrt{2\lambda}a)}+\dfrac{(a-r)\sinh(\sqrt{2\lambda}r)}{a\sinh(\sqrt{2\lambda}a)}
\right)dr\\ 
& =\dfrac{1}{2a}\int_{S^1}\int_{S^1}\phi(a,\theta_1)\psi(a,\theta_2)\eta(d\theta_1)\eta(d\theta_2). 
\end{split}
\]
This indicates $U$ is supported on $\partial F\times \partial F \setminus \mathtt{d}$ and for $x=(r,\theta), y=(r',\theta')$ with $x\neq y$, 
\begin{equation}\label{EQ3UXY}
U(dx,dy)=\frac{1}{2a}\delta_{a}(dr)\delta_{a}(dr')\eta(d\theta)\eta(d\theta').
\end{equation}
Therefore, the expression of $\check{\EE}$ follows from \eqref{TR}, \eqref{EM} and \eqref{EQ3UXY}.

Finally, one can easily find that $\mathbf{T}_a(W^\s)$ is associated with $(\check{\EE},\check{\FF})$ in the light of Theorem~\ref{PRO314}. That completes the proof. 
\end{proof}

\section{Stiff problem related to Walsh's Brownian motion}\label{SEC5}

In this section, we shall study the stiff problem related to the Walsh's Brownian motion and build a phase transition for it. 

\subsection{Mosco convergence}\label{MOS}
As in \cite{LS18}, we shall use Mosco convergence to describe the phase transition. For readers' convenience, we repeat its definition for handy reference. More details are referred to \cite{U94}. 

Let $(\EE^n,\FF^n)$ and $(\EE,\FF)$ be closed forms on $L^2(E,m)$, and we extend the domains of $\EE$ and $\EE^n$ to $L^2(E,m)$ by letting
\[
\begin{aligned}
	\EE(u,u)&:=\infty, \quad u\in L^2(E,m)\setminus \FF, \\ 
	\EE^n(u,u)&:=\infty,\quad u\in L^2(E,m)\setminus \FF^n.
\end{aligned}
\]
Then $(\EE^n,\FF^n)$ is said to be convergent to $(\EE,\FF)$ in the sense of Mosco as $n\rightarrow\infty$, if
\begin{itemize}
\item[(1)] For any sequence $\{u_n:n\geq 1\}\subset L^2(E,m)$ that converges weakly to $u$ in $L^2(E,m)$, it holds that
\[
	\EE(u,u)\leq \liminf_{n\rightarrow \infty}\EE^n(u_n,u_n). 
\]
\item[(2)] For any $u\in L^2(E,m)$, there exists a sequence $\{u_n:n\geq 1\}\subset L^2(E,m)$ that converges strongly to $u$ in $L^2(E,m)$ such that
\[
	\EE(u,u)\geq \limsup_{n\rightarrow \infty}\EE^n(u_n,u_n). 
\]
\end{itemize}
Here we say $u_n$ converges to $u$ weakly in $L^2(E,m)$, if for any $v\in L^2(E,m)$, $(u_n,v)_{L^2(E,m)}\rightarrow (u,v)_{L^2(E,m)}$ as $n\rightarrow \infty$, and strongly in $L^2(E,m)$, if $\|u_n-u\|_{L^2(E,m)}\rightarrow \infty$. The notations $(\cdot, \cdot)_{L^2(E,m)}$ and $\|\cdot\|_{L^2(E,m)}$ stand for the inner product and norm of $L^2(E,m)$. 

\subsection{Phase transition of stiff problem}

The stiff problem related to the WBM is described as follows. For $\varepsilon>0$, let $b_\varepsilon$ be a function on $[0,\varepsilon)$ such that for two constants $\delta_\varepsilon, C_\varepsilon>0$, 
\[
	\delta_\varepsilon\leq b_\varepsilon(r) \leq C_\varepsilon,\quad \text{a.e. }r\in [0,\varepsilon).  
\]
For any $f,g\in H^1(\mathbb{R})$, set 
\[
	\mathbf{D}^\varepsilon(u,v):= \int_0^\varepsilon b_\varepsilon(r) u'(r)v'(r)dr+\int_\varepsilon^\infty u'(r)v'(r)dr
\]
and define
\[
\begin{aligned}
	\FF^\varepsilon &:=\FF^W,\\
	\EE^\varepsilon(f,g) &:=\frac{1}{2}\int_{S^1}\mathbf{D}^\varepsilon(f_\theta, g_\theta)\eta(d\theta),\quad f,g\in \FF,
\end{aligned}
\]
where $f_\theta(\cdot):=f(\cdot, \theta), g_\theta(\cdot):=g(\cdot, \theta)$ as in \S\ref{SEC33}. Note that for any $f\in \FF^\varepsilon=\FF^W$, 
\[
	\delta_\varepsilon\wedge 1\cdot \EE^W_1(f,f)\leq \EE^\varepsilon_1(f,f)\leq C_\varepsilon\vee 1\cdot \EE^W_1(f,f). 
\]
This implies $(\EE^\varepsilon, \FF^\varepsilon)$ is regular on $L^2(\mathbb{R}^2,m)$. Roughly speaking, $(\EE^\varepsilon,\FF^\varepsilon)$ is obtained by attaching a small barrier at $B(0,\varepsilon):=\{x\in \mathbb{R}^2: |x|<\varepsilon\}$ to WBM. Then the stiff problem is concerned with the convergence of $(\EE^\varepsilon,\FF^\varepsilon)$ as $\varepsilon\downarrow 0$. This is also the main purpose of this section. 

Before moving on to the principal theorem, we prepare some notations.
Take a decreasing sequence $\varepsilon_n\downarrow 0$ and write $b_n, (\EE^n,\FF^n)$ for $b_{\varepsilon_n}, (\EE^{\varepsilon_n},\FF^{\varepsilon_n})$. Set 
\[
	\bar{\gamma}(n):=\int_0^{\varepsilon_n}\frac{1}{b_{\varepsilon_n}(r)}dr. 
\]
This parameter plays the role of total thermal resistance of the barrier $B(0,\varepsilon_n)$ as explained in \cite[Remark~4.3 and \S4.4]{LS18}.
The following result builds a phase transition for the stiff problem related to WBM. It is worth noting that the technical condition appeared in \cite[Theorem~4.6]{LS18} is not imposed.

\begin{theorem}\label{THM51}
Let $\varepsilon_n, b_n, (\EE^n,\FF^n)$ and $\bar{\gamma}(n)$ be given above. Assume 
\[
	\bar{\gamma}:=\lim_{n\rightarrow \infty}\bar{\gamma}(n)\quad (\leq \infty)
\]
exists. Then the following assertions hold:
\begin{itemize}
\item[(1)] $\bar{\gamma}=\infty$: $(\EE^n, \FF^n)$ converges to the Dirichlet form $(\EE, \FF)$ of reflecting WBM on $\mathbb{G}^2$ given by \eqref{EQ4FFL} in the sense of Mosco.
\item[(2)] $0<\bar{\gamma}<\infty$: $(\EE^n, \FF^n)$ converges to the Dirichlet form $(\EE^\s, \FF^\s)$ of SNOWB on $\mathbb{G}^2$ given by \eqref{EQ3FSFF} with the parameter $\kappa=(2\bar{\gamma})^{-1}$ in the sense of Mosco.
\item[(3)] $\bar{\gamma}=0$: $(\EE^n, \FF^n)$ converges to the Dirichlet form $(\EE^W, \FF^W)$ of WBM on $\mathbb{R}^2$ given by \eqref{EQ3FWF} in the sense of Mosco. 
\end{itemize}
\end{theorem}

\begin{proof}
The idea of the proof stems from those of Theorem 4.6 and Corollary 4.8 in \cite{LS18}. Let $(\EE^\dagger, \FF^\dagger)$ be one of $(\EE,\FF)$, $(\EE^\s,\FF^\s)$ and $(\EE^W,\FF^W)$. Write $H=L^2(\mathbb{R}^2,m)=L^2(\mathbb{G}^2,m)$, $a_n(r):=b_n(r)$ for $r\in [0,\varepsilon_n)$ and $a_n(r)=1$ for $r\geq \varepsilon_n$. 

We first prove the assertions under the assumption $\lim_{n\rightarrow\infty}\varepsilon_n\bar{\gamma}(n)=0$. To show the first part of Mosco convergence in \S\ref{MOS}, suppose $\{f^n\}$ converges to $f$ weakly in $H$ and 
\[
	\liminf_{n\rightarrow\infty}\EE^n(f^n,f^n)\leq \sup_{n\geq 1}\EE^n(f^n,f^n)=:M<\infty.
\]
Recall that $\mathbf{T}_n:=\mathbf{T}_{\varepsilon_n}: \mathbb{G}^2\rightarrow \mathbb{R}^2\setminus B(0,\varepsilon_n)$ is a homeomorphism. Set $\breve{f}^n:=f^n\circ \mathbf{T}_n$, i.e. $\breve{f}^n(r,\theta):=f^n(r+\varepsilon_n, \theta)$ for any $(r,\theta)\in \mathbb{G}^2$. We claim $\|f^n-\breve{f}^n\|_H\rightarrow 0$ as $n\rightarrow \infty$ and particularly, $\breve{f}^n$ converges to $f$ weakly in $H$. Indeed,  
\[
\begin{aligned}
\|f^n-\breve{f}^n\|_H^2 &=\int_{S^1}\eta(d\theta)\int_0^\infty \left(f^n_\theta(r+\varepsilon_n)-f^n_\theta(r)\right)^2dr \\
&=\int_{S^1}\eta(d\theta)\int_0^\infty \left(\int_r^{r+\varepsilon_n} 
\nabla f^n_\theta(\rho)d\rho\right)^2dr \\
&\leq \int_{S^1}\eta(d\theta)\int_0^\infty \left(\int_r^{r+\varepsilon_n} a_n(\rho)
\nabla f^n_\theta(\rho)^2d\rho\right)\cdot\left(\int_r^{r+\varepsilon_n}\frac{1}{a_n(\rho)}d\rho\right) dr \\
&\leq (\bar{\gamma}(n)+\varepsilon_n)\int_{S^1}\eta(d\theta)\int_0^\infty \left(\int_r^{r+\varepsilon_n} a_n(\rho)
\nabla f^n_\theta(\rho)^2d\rho\right) dr \\
&\leq (\bar{\gamma}(n)+\varepsilon_n)\int_{S^1}\eta(d\theta)\int_0^\infty a_n(\rho)
\nabla f^n_\theta(\rho)^2d\rho \int_{(\rho-\varepsilon_n)\vee 0}^\rho dr\\
&\leq 2\varepsilon_n(\bar{\gamma}(n)+\varepsilon_n)M.
\end{aligned}
\]
Then it follows from $\lim_{n\rightarrow \infty}\varepsilon_n \bar{\gamma}(n)=0$ that $\|f^n-\breve{f}^n\|_H\rightarrow 0$. Now we prove $\EE^\dagger(f,f)\leq \liminf_{n\rightarrow \infty}\EE^n(f^n,f^n)$ for the three cases respectively.
\begin{itemize}
\item[(1)] $\bar{\gamma}=\infty$: Clearly $\EE(f,f)\leq \liminf_{n\rightarrow \infty}\EE(\breve{f}^n,\breve{f}^n)\leq \liminf_{n\rightarrow \infty}\EE^n(f^n,f^n)$.  
\item[(2)] $0<\bar{\gamma}<\infty$: Note that
\[
	\EE^\s(\breve{f}^n,\breve{f}^n)=\frac{1}{2}\int_{S^1}\int_{\varepsilon_n}^\infty (f^n_\theta)'(r)^2dr\eta(d\theta)+\frac{\kappa}{2}\int\left(f^n_{\theta_1}(\varepsilon_n)-f^n_{\theta_2}(\varepsilon_n)\right)^2\eta(d\theta_1)\eta(d\theta_2). 
\]
Since $f_{\theta_1}^n(0)=f_{\theta_2}^n(0)$, it follows that
\begin{equation}\label{EQ5SSF}
\begin{split}
\dfrac{1}{2}\int_{S^1\times S^1}&\left(f^n_{\theta_1}(\varepsilon_n)-f^n_{\theta_2}(\varepsilon_n)\right)^2\eta(d\theta_1)\eta(d\theta_2)\\
&=\dfrac{1}{2}\int_{S^1\times S^1}\left(\int_0^{\varepsilon_n}(f^n_{\theta_1})'(r)dr-\int_0^{\varepsilon_n}(f^n_{\theta_2})'(r)dr\right)^2\eta(d\theta_1)\eta(d\theta_2)\\
&\leq\int_{S^1}\left(\int_0^{\varepsilon_n}(f^n_{\theta})'(r)dr\right)^2\eta(d\theta) \\
&\leq \bar{\gamma}(n)\int_{S^1}\int^{\varepsilon_n}_{0}b_n(x)(f^n_\theta)'(r)^2dr\eta(d\theta).
\end{split}
\end{equation}
Thus we can conclude $\EE^\s(f,f)\leq \liminf_{n\rightarrow \infty}\EE^\s(\breve{f}^n,\breve{f}^n)\leq \liminf_{n\rightarrow \infty}\EE^n(f^n,f^n)$ from $\kappa \cdot \bar{\gamma}_n\rightarrow 1/2$.
\item[(3)] $\bar{\gamma}=0$: 
Since $\sup_{n}\EE(\breve{f}^n,\breve{f}^n)\leq \sup_n \EE^n(f^n,f^n)\leq M$, and the weak convergence of $\breve{f}_n$ in $H$ implies $\sup_n\|\breve{f}^n\|_H<\infty$, it follows that  $\sup_{n}\EE_1(\breve{f}^n,\breve{f}^n)<\infty$. By Banach-Saks theorem, take a subsequence if necessary, the Ces\`aro mean of $\{\breve{f}^n\}$ converges to some $h\in \FF$ in $\|\cdot \|_{\EE_1}$-norm. Then $h_k:=\frac{1}{k}\sum_{n=1}^k\breve{f}_n$ is $\EE_1$-convergent to $h$. That means $h_k$ converges to $h$, $\EE$-q.e. We claim that $h=f$. Indeed, take any $u\in H$, we have $(\breve{f}_n, u)_H\rightarrow (f, u)_H$ and 
\[
(h,u)_H=\lim_{k\rightarrow \infty} (h_k, u)_H=\lim_{k\rightarrow \infty}\frac{1}{k}\sum_{n=1}^k (\breve{f}_n, u)_H=(f,u)_H.
\]
We have proved in the case $\bar{\gamma}=\infty$ that $f\in\FF$ and $\EE(f,f)\leq \liminf_{n\rightarrow \infty}\EE^n(f^n,f^n)$, and it suffices to show $f\in \FF^W$. Note that if $A\subset \{0\}\times S^1$ is $\EE$-polar, then $\left(\delta_0\times \eta\right)(A)=0$ by Lemma~\ref{LM313}. Thus $f(0,\cdot)$ is $\eta$-a.e. defined on $S^1$. Let
\[
\begin{aligned}
&c_\#:=\inf\left\{c\in \mathbb{R}: \eta(f(0,\cdot)>c)=0\right\}, \\
&c^\#:=\sup\left\{c\in \mathbb{R}: \eta(f(0,\cdot)<c)=0\right\}. 
\end{aligned}\] 
Clearly, $c^\#\leq f(0,\cdot)\leq c_\#$, $\eta$-a.e. We need only show $c^\#=c_\#$. Suppose $c^\#<c_\#$. Take $c^\#<c<c_\#$ and we have $\eta(f(0,\cdot)>c)>0, \eta(f(0,\cdot)<c)>0$. This implies 
\begin{equation}\label{EQ5SSFT}
	\int_{S^1\times S^1} \left(f(0,\theta_1)-f(0,\theta_2)\right)^2\eta(d\theta_1)\eta(d\theta_2)>0. 
\end{equation}
However, by Fatou lemma we obtain
\[
\begin{aligned}
\int &\left(f(0,\theta_1)-f(0,\theta_2)\right)^2\eta(d\theta_1)\eta(d\theta_2) \\
&\qquad=\int \lim_{k\rightarrow \infty}\left(h_k(0,\theta_1)-h_k(0,\theta_2)\right)^2\eta(d\theta_1)\eta(d\theta_2) \\
&\qquad \leq \liminf_{k\rightarrow \infty}\frac{1}{k}\sum_{n=1}^k\int \left(\breve{f}^n(0,\theta_1)-\breve{f}^n(0,\theta_2) \right)^2 \eta(d\theta_1)\eta(d\theta_2).
\end{aligned}
\]
It follows from \eqref{EQ5SSF} that
\[
\begin{aligned}
	\int &\left(\breve{f}^n(0,\theta_1)-\breve{f}^n(0,\theta_2) \right)^2 \eta(d\theta_1)\eta(d\theta_2) \\
	&\qquad = \int \left({f}^n(\varepsilon_n,\theta_1)-{f}^n(\varepsilon_n,\theta_2) \right)^2 \eta(d\theta_1)\eta(d\theta_2) \\
	&\qquad \leq 2M\bar{\gamma}(n) \\
	&\qquad \rightarrow 0. 
	\end{aligned}\]
Hence $\int \left(f(0,\theta_1)-f(0,\theta_2)\right)^2\eta(d\theta_1)\eta(d\theta_2)=0$, which contradicts \eqref{EQ5SSFT}. 
\end{itemize}
To prove the second part of Mosco convergence, let $g\in H$ with $\EE^\dagger(g,g)<\infty$. Denote $c:=\int_{S^1}g(0,\theta)\eta(d\theta)$ and define a function $g^n\in \FF^n$ as follows: For any $r\geq \varepsilon_n$, set $g^n(r,\theta):=g(r-\varepsilon_n,\theta), \forall \theta\in S^1$ and for $r\in [0,\varepsilon_n), \theta\in S^1$, set
\[
	g^n(r,\theta):=c+\frac{g(0, \theta)-c}{\bar{\gamma}(n)}\int_0^r\frac{1}{b_n(\rho)}d\rho.
\]  
It is easy to see $g^n\rightarrow g$ strongly in $H$ and we have
\[
\EE^n(g^n,g^n)=\frac{1}{2}\int_{S^1}\int_0^\infty g'_\theta(r)^2dr\eta(d\theta)+\frac{1}{2\bar{\gamma}(n)}\int_{S^1}(g(0, \theta)-c)^2\eta(d\theta). 
\]
Note that 
\[
	2\int_{S^1}(g(0, \theta)-c)^2\eta(d\theta)=\int_{S^1\times S^1}(g(0,\theta_1)-g(0,\theta_2))^2\eta(d\theta_1)\eta(d\theta_2)
\]
on account of the notation $c=\int g(0,\theta)\eta(d\theta)$. Then for the case $\bar{\gamma}=\infty$ or $0<\bar{\gamma}<\infty$, we can conclude that $\lim_{n\rightarrow \infty}\EE^n(g^n,g^n)=\EE(g,g)$ or $\EE^\s(g,g)$ respectively. For the case $\bar{\gamma}=0$, it suffices to note that $g\in \FF^W$ implies $g(0,\theta)=c$ and thus $\EE^n(g^n,g^n)=\EE^W(g,g)$.

Finally, we prove the case $\bar{\gamma}=\infty$ without the assumption $\lim_{n\rightarrow\infty}\varepsilon_n\bar{\gamma}(n)=0$. On one hand, suppose $\{f^n\}$ converges to $f$ weakly in $H$ and $$\liminf_{n\rightarrow\infty}\EE^n(f^n,f^n)\leq \sup_{n\geq 1}\EE^n(f^n,f^n)=:M<\infty.$$ Set
\[
\breve{f}^n|_{B(0,\varepsilon_n)^c}:= f^n|_{B(0,\varepsilon_n)^c},\quad \breve{f}^n(r,\theta):=f^n(\varepsilon_n, \theta),\quad r\in [0+,\varepsilon_n), \theta\in S^1. 
\]
Clearly,  $\breve{f}^n\in \FF$, and we claim that $\breve{f}^n$ converge to $f$ weakly in $H$. Indeed, for any $g\in H$, we have
\[
\left(f^n-\breve{f}^n, g\right)_H=\int_{B_n}f^n_{\theta}(r)g_{\theta}(r)dr\eta(d\theta)-\int_{B_n}f^n_{\theta}(\varepsilon_n)g_{\theta}(r)dr\eta(d\theta).
\]
The weak convergence of $\{f^n\}$ implies $K:=\sup_{n}\|f^n\|^2_{H}<\infty$. Thus as $n\rightarrow \infty$, 
\[
	\left|\int_{B_n} f^n_{\theta}(r)g_{\theta}(r)dr\eta(d\theta)\right|^2\leq K\cdot \int_{B_n}g_{\theta}(r)^2dr\eta(d\theta)\rightarrow 0. 
\]
Since $f^n\in\FF^n$, $f^n_{\theta}(r)$ is absolutely continuous on $[\varepsilon_n,\infty)$ for $\theta\in A\subset S^1$ with $\eta(A)=1$. For any $\theta\in A$ and $r>\varepsilon_n$, we have
$|f^n_{\theta}(\varepsilon_n)|\leq |f^n_{\theta}(r)|+|\int_{\varepsilon_n}^r (f^n_{\theta})'(y)dy|$ and
\begin{equation}\label{EQ4FNN}
\begin{split}
	\int_{S^1}|f^n_{\theta}(\varepsilon_n)|^2\eta(d\theta)
	& \leq 2\int_{S^1}\int_{\varepsilon_n}^{1+\varepsilon_n} f^n_{\theta}(r)^2dr\eta(d\theta)+2\int_{S^1}\int_{\varepsilon_n}^\infty  (f^n_{\theta})'(y)^2dy\eta(d\theta)\\
	& \leq 4\EE^n_1(f^n,f^n). 
\end{split}
\end{equation}
This implies as $n\rightarrow 0$,
\[
\left|\int_{B_n}f^n_{\theta}(\varepsilon_n)g_{\theta}(r)dr\eta(d\theta)\right|^2\leq 4(M+K)\cdot \int_{B_n}g_{\theta}(r)^2dr\eta(d\theta)\rightarrow 0. 
\]
Thus we can conclude that $\breve{f}^n$ converge to $f$ weakly in $H$. As a result,
\[
\EE(f,f)\leq \liminf_{n\rightarrow \infty}\EE(\breve{f}^n,\breve{f}^n)\leq \liminf_{n\rightarrow \infty}\EE^n(f^n,f^n). 
\]
On the other hand, let $g\in H$ with $\EE(g,g)<\infty$. Denote $c_n:=\int_{S^1}g(\varepsilon_n,\theta)\eta(d\theta)$ and for each $n$, define
\[
g^n|_{B(0,\varepsilon_n)^c}:= g|_{B(0,\varepsilon_n)^c},\quad g^n(r,\theta):=c_n+\frac{g(\varepsilon_n,\theta)-c_n}{\bar{\gamma}(n)}\int_0^r\frac{1}{b_n(\rho)}d\rho. 
\]
We assert that $\|g^n-g\|_H\rightarrow 0$ as $n\rightarrow 0$. In fact, for any $r\in(0,\varepsilon_n]$, $|g_{\theta}^n(r)|^2\leq 2c_n^2+2\left(g_{\theta}(\varepsilon_n)-c_n\right)^2$, so that
\[
\int_{S^1}|g_{\theta}^n(r)|^2\eta(d\theta)\leq 6\int_{S^1}|g_{\theta}(\varepsilon_n)|^2\eta(d\theta)\leq 24 \EE_1(g,g).
\] 
The last inequality is similar to \eqref{EQ4FNN}, and it follows that
\[
\begin{split}
\|g^n-g\|_H
& \leq 2\int_{B_n} g_{\theta}(r)^2dr\eta(d\theta)+2\int_0^{\varepsilon_n}\int_{S^1}|g_{\theta}^n(r)|^2\eta(d\theta)dr\\
& \leq 2\int_{B_n} g_{\theta}(r)^2dr\eta(d\theta)+48\EE_1(g,g)\cdot\varepsilon_n\rightarrow 0
\end{split}
\]
as $n\rightarrow\infty$. 
Then we have
\[
\begin{split}
	\EE^n(g^n,g^n)
	& =\frac{1}{2}\int_{S^1}\int_{\varepsilon_n}^\infty g'_\theta(r)^2dr\eta(d\theta)+\frac{1}{2\bar{\gamma}(n)}\int_{S^1}(g(\varepsilon_n, \theta)-c)^2\eta(d\theta)\\
	& \leq \frac{1}{2}\int_{S^1}\int_{\varepsilon_n}^\infty g'_\theta(r)^2dr\eta(d\theta)+\frac{2}{\bar{\gamma}(n)}\int_{S^1}|g_{\theta}(\varepsilon_n)|^2\eta(d\theta)\\
	&\leq \frac{1}{2}\int_{S^1}\int_{\varepsilon_n}^\infty g'_\theta(r)^2dr\eta(d\theta)+\frac{8}{\bar{\gamma}(n)}\EE_1(g,g).
\end{split}
\]
Since $\bar{\gamma}(n)\rightarrow \infty$, we can conclude
\[
	\limsup_{n\rightarrow \infty}\EE^n(g^n,g^n)=\limsup_{n\rightarrow \infty}\frac{1}{2}\int_{S^1}\int_{\varepsilon_n}^\infty g'_\theta(r)^2dr\eta(d\theta)\leq \EE(g,g). 
\]
That completes the proof. 
\end{proof}

\begin{remark}
If we take $b_\varepsilon(r):=(\kappa\varepsilon)^{-\alpha}$ for any $r\in [0,\varepsilon)$ with a fixed parameter $\kappa>0$, then the three phases in Theorem~\ref{THM51} correspond to $\alpha<-1$, $\alpha=-1$ and $\alpha>-1$ respectively. 
\end{remark}

\subsection{Continuity of phase transition}

We complete this section with a result, which states the continuity of phase transition in Theorem~\ref{THM51}. This continuity was considered for the stiff problems in one-dimensional space by means of Mosco convergence of Dirichlet forms in \cite{LS18}. 
To explain it, denote the Dirichlet form of the phase with parameter $\bar{\gamma}$ in Theorem~\ref{THM51} by $(\EE^{\bar{\gamma}},\FF^{\bar{\gamma}})$. More precisely,
\begin{itemize}
\item[(1)] $\bar{\gamma}=\infty$: $(\EE^\infty,\FF^\infty):=(\EE,\FF)$ given by \eqref{EQ4FFL};
\item[(2)] $\bar{\gamma}\in(0,\infty)$: $(\EE^{\bar{\gamma}},\FF^{\bar{\gamma}}):=(\EE^\s,\FF^\s)$ given by \eqref{EQ3FSFF} with the parameter $\kappa=(2\bar{\gamma})^{-1}$;
\item[(3)] $\bar{\gamma}=0$:  $(\EE^0,\FF^0):=(\EE^W,\FF^W)$ given by \eqref{EQ3FWF}. 
\end{itemize}  
The following theorem is an analogical result of \cite[Theorem~4.11]{LS18}. 
 

\begin{theorem}\label{THM43}
Let $\{\bar{\gamma}_n: n\geq 1\}$ be a sequence in $[0,\infty]$ such that $$\lim_{n\rightarrow\infty}\bar{\gamma}_n=\bar{\gamma}\in [0,\infty]. $$ Then $(\EE^{\bar{\gamma}_n},\FF^{\bar{\gamma}_n})$ converges to $(\EE^{\bar{\gamma}},\FF^{\bar{\gamma}})$ in the sense of Mosco as $n\rightarrow \infty$. 
\end{theorem}
\begin{proof}
Write $H=L^2(\mathbb{G}^2,m)=L^2(\mathbb{R}^2,m)$. 
Without loss of generality, we assume $0<\bar{\gamma}_n<\infty$ for any $n\geq 1$. 
As we have phrased in Theorem \ref{THM317}, $(\EE^0, \FF^0)$ is the darning of SNOWB by shorting $\{0+\}\times S^1$ into $\bf{0}$. The assertion of the case $\bar{\gamma}=0$ is implied by \cite[Theorem 4.3]{CP17}. In what follows, we shall consider the case $0<\bar{\gamma}\leq\infty$. 

Let us prove the first item in the definition of Mosco convergence. Assume $\{u_n: n\geq 1\}\subset H$ converges weakly to $u$ in $H$ and 
\[
	\liminf_{n\rightarrow \infty} \EE^{\bar{\gamma}_n}(u_n,u_n)\leq \sup_n  \EE^{\bar{\gamma}_n}(u_n,u_n)=:M<\infty. 
\]
For the case $\bar{\gamma}=\infty$, we have
\[
	\EE^\infty(u,u)\leq \liminf_{n\rightarrow \infty}\EE^\infty(u_n,u_n)\leq \liminf_{n\rightarrow \infty} \EE^{\bar{\gamma}_n}(u_n,u_n).
\]
For the case $0<\bar{\gamma}<\infty$, take a constant $\bar{\gamma}<K<\infty$. Then there exists an integer $N$ such that for any $n>N$, $\bar{\gamma}_n<K$. This leads to
\[
\mathscr{U}_n:=\int \left(u_n(0,\theta)-u_n(0,\theta') \right)^2\eta(d\theta)\eta(d\theta')\leq 4K\EE^{\bar{\gamma}_n}(u_n,u_n)\leq 4KM. 
\]
We have
\[
\EE^{\bar{\gamma}}(u,u) \leq \liminf_{n\rightarrow \infty}\EE^{\bar{\gamma}}(u_n,u_n) =\liminf_{n>N, n\rightarrow \infty}\left(\EE^{\bar{\gamma}_n}(u_n,u_n)+ \left(\frac{1}{4\bar{\gamma}}-\frac{1}{4\bar{\gamma}_n}\right)\cdot \mathscr{U}_n\right).
\]
Since the second term in the right-hand side is not greater than
\[
\left|\frac{1}{4\bar{\gamma}}-\frac{1}{4\bar{\gamma}_n}\right| \cdot 4KM\rightarrow 0 
\]
as $\bar{\gamma}_n\rightarrow \bar{\gamma}$, we obtain $\EE^{\bar{\gamma}}(u,u)\leq\liminf_{n>N, n\rightarrow \infty}\EE^{\bar{\gamma}_n}(u_n,u_n)=\liminf_{n\rightarrow \infty}\EE^{\bar{\gamma}_n}(u_n,u_n)$. 

To show the second item in the definition of Mosco convergence, let $u\in H$ be such that $\EE^{\bar{\gamma}}(u,u)<\infty$. This implies $u\in \FF^{\bar{\gamma}}=\FF^{\bar{\gamma}_n}$. We only need to take $u_n:=u$, since 
\[
\EE^{\bar{\gamma}_n}(u,u) =\EE^{\bar{\gamma}}(u,u)+ \left(\frac{1}{4\bar{\gamma}_n}-\frac{1}{4\bar{\gamma}}\right)\int \left(u(0,\theta)-u(0,\theta') \right)^2\eta(d\theta)\eta(d\theta')
\] 
and the second term in the right-hand side converges to $0$ as $\bar{\gamma}_n\rightarrow \bar{\gamma}$. 
\end{proof}

\bibliographystyle{abbrv}
\bibliography{SNWBM}

\end{document}